\documentclass[11pt]{article}    

\usepackage[margin=0.75in]{geometry} 
\usepackage{amsmath,amssymb,amsthm}

\newtheorem{theorem}{Theorem}[section]

\newtheorem{lemma}[theorem]{Lemma}

\newtheorem{remark}{Remark}

\newcommand{\al}{\alpha}
\newcommand{\bt}{\beta}
\newcommand{\la}{\lambda}
\newcommand{\s}{\sigma}
\newcommand{\be}{\begin{equation}}
\newcommand{\ee}{\end{equation}}
\newcommand{\bea}{\begin{eqnarray}}
\newcommand{\eea}{\end{eqnarray}}
\newcommand{\no}{\nonumber}

\numberwithin{equation}{section}
\begin{document}

\title{\Large Differential and Difference Equations for Recurrence Coefficients of Orthogonal Polynomials with a Singularly Perturbed Laguerre-type Weight}
\author{Chao Min, Yuan Cheng and Yang Chen}


\date{\today}
\maketitle
\begin{abstract}
We are concerned with the monic orthogonal polynomials with respect to a singularly perturbed Laguerre-type weight. By using the ladder operator approach, we derive a complicated system of nonlinear second-order difference equations satisfied by the recurrence coefficients. This allows us to derive the large $n$ asymptotic expansions of the recurrence coefficients. In addition, we also obtain a system of differential-difference equations for the recurrence coefficients.
\end{abstract}

$\mathbf{Keywords}$: Orthogonal polynomials; Recurrence coefficients; Ladder operator approach;

Singularly perturbed Laguerre-type weight; Difference equations; Asymptotic expansions.

$\mathbf{Mathematics\:\: Subject\:\: Classification\:\: 2020}$: 42C05, 41A60.

\section{Introduction}
It is well known that the recurrence coefficients of semi-classical orthogonal polynomials are often related to the solutions of Painlev\'{e} equations. Two typical examples are given below. Chen and Its \cite{ChenIts} prove that the diagonal recurrence coefficient of the monic orthogonal polynomials with a singularly perturbed Laguerre weight
\be\label{w1}
w(x)=x^{\al}\mathrm{e}^{-x-\frac{s}{x}},\qquad x\in \mathbb{R}^{+},\;\al>0,\;s>0
\ee
satisfies a particular third Painlev\'{e} equation. Filipuk, Van Assche and Zhang \cite{Filipuk} (see also \cite{Clarkson}) establish the relationship between recurrence coefficients of the semi-classical Laguerre polynomials with the weight
\be\label{w2}
w(x)=x^{\al}\mathrm{e}^{-x^2+tx},\qquad x\in \mathbb{R}^{+},\;\al>-1,\; t\in \mathbb{R}
\ee
and the fourth Painlev\'{e} equation. For more examples, see \cite{Clarkson3,Dai,Magnus,Min2021,VanAssche} and the references therein.

In view of (\ref{w1}) and (\ref{w2}), it is natural to consider the following singularly perturbed Laguerre-type weight
\be\label{wei}
w(x;t):=x^\la\mathrm{e}^{-x^2-\frac{t}{x}},\qquad x\in \mathbb{R}^{+}
\ee
with $\la\geq 0,\;\;t>0$. This is also a semi-classical weight since it satisfies the Pearson equation (see, e.g., \cite[Section 1.1.1]{VanAssche})
$$
\frac{d}{dx}(\s(x)w(x))=\tau(x)w(x),
$$
with $\s(x)=x^2,\; \tau(x)=-2x^3+(\la+2)x+t$.

The question is whether or not the recurrence coefficients for the orthogonal polynomials with the weight (\ref{wei}) have any connection to the Painlev\'{e} equations. By using the ladder operator approach, we obtain a system of second-order difference equations and also a system of differential-difference equations satisfied by the recurrence coefficients. We find that this problem is very complicated compared with \cite{ChenIts,Filipuk} and can not detect any connection to the Painlev\'{e} equations. Hence, our study contributes to a better understanding about how to choose a suitable weight function to establish the relation between recurrence coefficients of the corresponding orthogonal polynomials and Painlev\'{e} equations as shown in \cite{CFR}.
\section{Preliminaries and Main Results}
Let $P_{n}(x;t),\; n=0,1,2,\ldots$ be the monic polynomials of degree $n$ orthogonal with respect to the weight (\ref{wei}), i.e.,
\be\label{or}
\int_{0}^{\infty}P_{m}(x;t)P_{n}(x;t)w(x;t)dx=h_{n}(t)\delta_{mn},\qquad m, n=0,1,2,\ldots.
\ee
Here $P_{n}(x;t)$ can be written in the expansion form
\be\label{expan}
P_{n}(x;t)=x^{n}+\mathrm{p}(n,t)x^{n-1}+\cdots+P_n(0;t),\qquad n=0,1,2,\ldots,
\ee
where $\mathrm{p}(n,t)$ denotes the coefficient of $x^{n-1}$ with the initial value $\mathrm{p}(0,t)=0$.

One of the most important properties of orthogonal polynomials is that they satisfy the three-term recurrence relation \cite{Szego}
\be\label{rr}
xP_{n}(x;t)=P_{n+1}(x;t)+\al_n(t)P_n(x;t)+\beta_{n}(t)P_{n-1}(x;t),
\ee
supplemented by the initial conditions $P_{0}(x;t)=1,\; \beta_{0}(t)P_{-1}(x;t)=0.$

It follows from (\ref{or}), (\ref{expan}) and (\ref{rr}) that the following relations hold:
\be\label{al}
\al_n(t)=\mathrm{p}(n,t)-\mathrm{p}(n+1,t),
\ee
\be\label{be}
\beta_{n}(t)=\frac{h_{n}(t)}{h_{n-1}(t)}.
\ee

The ladder operator approach is very useful and powerful to analyze the recurrence coefficients of orthogonal polynomials.
Following the general set-up by Chen and Ismail \cite{Chen1997,ChenIsmail2005} (see also Ismail \cite[Chapter 3]{Ismail}), we have a pair of ladder operators for our orthogonal polynomials (the $t$-dependence of many quantities will not be displayed for brevity from now on):
$$
\left(\frac{d}{dz}+B_{n}(z)\right)P_{n}(z)=\beta_{n}A_{n}(z)P_{n-1}(z),
$$
$$
\left(\frac{d}{dz}-B_{n}(z)-\mathrm{v}'(z)\right)P_{n-1}(z)=-A_{n-1}(z)P_{n}(z),
$$
where $\mathrm{v}(z):=-\ln w(z)$ is the potential and
\be\label{an}
A_{n}(z):=\frac{1}{h_{n}}\int_{0}^{\infty}\frac{\mathrm{v}'(z)-\mathrm{v}'(y)}{z-y}P_{n}^{2}(y)w(y)dy,
\ee
\be\label{bn}
B_{n}(z):=\frac{1}{h_{n-1}}\int_{0}^{\infty}\frac{\mathrm{v}'(z)-\mathrm{v}'(y)}{z-y}P_{n}(y)P_{n-1}(y)w(y)dy.
\ee
The associated compatibility conditions for $A_n(z)$ and $B_n(z)$ are
\be
B_{n+1}(z)+B_{n}(z)=(z-\al_n) A_{n}(z)-\mathrm{v}'(z), \tag{$S_{1}$}
\ee
\be
1+(z-\al_n)(B_{n+1}(z)-B_{n}(z))=\beta_{n+1}A_{n+1}(z)-\beta_{n}A_{n-1}(z), \tag{$S_{2}$}
\ee
\be
B_{n}^{2}(z)+\mathrm{v}'(z)B_{n}(z)+\sum_{j=0}^{n-1}A_{j}(z)=\beta_{n}A_{n}(z)A_{n-1}(z), \tag{$S_{2}'$}
\ee
where ($S_{2}'$) is obtained from the combination of ($S_{1}$) and ($S_{2}$) and usually gives better insight into the recurrence coefficients.

Our main results of this paper are as follows.
\begin{theorem}\label{thm}
The recurrence coefficients $\al_n$ and $\bt_n$ satisfy the following system of nonlinear \textbf{second}-order difference equations:
\begin{subequations}\label{de}
\begin{small}
\bea\label{d1}
&&\big[n t-2t\bt_n-2\al_n \beta_n(2\alpha_{n-1}^2+2 \beta_n+2 \beta_{n-1}-2 n+1-\la)
-2\alpha_{n-1}\beta_n (2 \alpha_n^2+2 \beta_n+2 \beta_{n+1} -2 n-1-\la)\big]\no\\
&&\times\big[(n+\la) t-2t\bt_n+2\al_n \beta_n(2\alpha_{n-1}^2+2 \beta_n+2 \beta_{n-1}-2 n+1-\la)
+2\alpha_{n-1}\beta_n (2 \alpha_n^2+2 \beta_n+2 \beta_{n+1} -2 n-1-\la)\big]\no\\
&&+\bt_n(2n+\la-4\bt_n)^2(2\alpha_{n-1}^2+2 \beta_n+2 \beta_{n-1}-2 n+1-\la)(2 \alpha_n^2+2 \beta_n+2 \beta_{n+1} -2 n-1-\la)=0,
\eea
\end{small}\\[-45pt]
\begin{small}
\bea\label{d2}
&&2\big[n t-2t\bt_n-2\al_n \beta_n(2\alpha_{n-1}^2+2 \beta_n+2 \beta_{n-1}-2 n+1-\la)
-2\alpha_{n-1}\beta_n (2 \alpha_n^2+2 \beta_n+2 \beta_{n+1} -2 n-1-\la)\big]\no\\
&&+(2n+\la-4\bt_n)\big[2\al_n^3-(2n+2+\la)\al_n-t+4\al_n\bt_{n+1}-2\al_{n-1}\bt_n+2\al_{n+1}\bt_{n+1}\big]=0.
\eea
\end{small}
\end{subequations}
\end{theorem}
\begin{theorem}\label{thm1}
The recurrence coefficients $\al_n$ and $\bt_n$ satisfy the system of differential-difference equations\\[-15pt]
\begin{subequations}\label{dd}
\be\label{dd1}
t\al_n'(t)=\al_n+2\bt_n(\al_n+\al_{n-1})-2\bt_{n+1}(\al_n+\al_{n+1}),
\ee
\be\label{dd2}
t\bt_n'(t)=2\bt_n(\al_{n-1}^2-\al_n^2+\bt_{n-1}-\bt_{n+1}+1).
\ee
\end{subequations}
\end{theorem}
\begin{theorem}\label{thm2}
The recurrence coefficients $\al_n$ and $\bt_n$ have the following asymptotic expansions as $n\rightarrow\infty$:
$$
\al_n=\sqrt{\frac{2n}{3}}+\frac{2 \lambda+2 +3 \sqrt[3]{2t^2}}{4 \sqrt{6n}}-\frac{4 \sqrt[3]{2} (3 \lambda +2)+18(\lambda +1) (2t)^{2/3}+27 t^{4/3}}{96\sqrt{3}\times 2^{5/6}  n^{3/2}}+\frac{\lambda (\lambda ^2-1)}{72\times\sqrt[3]{4t}\:n^2}+O(n^{-5/2}),
$$\\[-30pt]
$$
\bt_n=\frac{n}{6}+\frac{\lambda }{12}-\frac{\lambda  \sqrt[3]{t}}{4\times2^{5/6} \sqrt{3n}}+\frac{1-3 \lambda ^2}{144 n}-\frac{\lambda  \left[2^{2/3} (\lambda ^2-1)-6 \lambda  t^{2/3}-9 \sqrt[3]{2}\: t^{4/3}\right]}{96\sqrt{3}\times2^{5/6}   \sqrt[3]{t}\:n^{3/2}}+O(n^{-2}).
$$
\end{theorem}
\section{Proof of the Main Results}
We mainly follow what was developed in \cite{Chen1997,ChenIsmail2005,Ismail}. For our weight (\ref{wei}), we have
$$
\mathrm{v}(z)=-\ln w(z)=z^2+\frac{t}{z}-\lambda\ln z.
$$
It follows that
$$
\mathrm{v}'(z)=2z-\frac{\lambda}{z}-\frac{t}{z^2}
$$
and
\be\label{vp}
\frac{\mathrm{v}'(z)-\mathrm{v}'(y)}{z-y}
=2+\frac{\lambda}{zy}+\frac{t}{zy^2}+\frac{t}{z^2y}.
\ee
\begin{lemma}
We have
\be\label{an1}
A_{n}(z)=2+\frac{2\al_{n}}{z}+\frac{R_{n}(t)}{z^2},
\ee
\be\label{bn1}
B_{n}(z)=\frac{2\beta_{n}-n}{z}+\frac{r_{n}(t)}{z^2},
\ee
where $R_{n}(t)$ and $r_{n}(t)$ are the auxiliary quantities given by
$$
R_{n}(t):=\frac{t}{h_{n}}\int_{0}^{\infty}\frac{1}{y}P_{n}^{2}(y)w(y)dy,
$$
$$
r_{n}(t):=\frac{t}{h_{n-1}}\int_{0}^{\infty}\frac{1}{y}P_{n}(y)P_{n-1}(y)w(y)dy.
$$
\end{lemma}
\begin{proof}
Substituting (\ref{vp}) into the definitions of $A_n(z)$ and $B_n(z)$ in (\ref{an}) and (\ref{bn}), we have
\be\label{an2}
A_n(z)=2+\left(\frac{\la}{zh_n}+\frac{t}{z^2h_n}\right)\int_{0}^{\infty}\frac{1}{y}P_{n}^{2}(y)w(y)dy
+\frac{t}{zh_n}\int_{0}^{\infty}\frac{1}{y^2}P_{n}^{2}(y)w(y)dy,
\ee
\be\label{bn2}
B_n(z)=\left(\frac{\la}{zh_{n-1}}+\frac{t}{z^2h_{n-1}}\right)\int_{0}^{\infty}\frac{1}{y}P_{n}(y)P_{n-1}(y)w(y)dy
+\frac{t}{zh_{n-1}}\int_{0}^{\infty}\frac{1}{y^2}P_{n}(y)P_{n-1}(y)w(y)dy.
\ee
Using integration by parts, we find
\be\label{e1}
\frac{\la}{h_n}\int_{0}^{\infty}\frac{1}{y}P_{n}^{2}(y)w(y)dy=2\al_n-\frac{t}{h_n}\int_{0}^{\infty}\frac{1}{y^2}P_{n}^{2}(y)w(y)dy,
\ee
\be\label{e2}
\frac{\la}{h_{n-1}}\int_{0}^{\infty}\frac{1}{y}P_{n}(y)P_{n-1}(y)w(y)dy=-n+2\bt_n-\frac{t}{h_{n-1}}\int_{0}^{\infty}\frac{1}{y^2}P_{n}(y)P_{n-1}(y)w(y)dy,
\ee
where use has been made of (\ref{rr}) and (\ref{be}). The lemma follows by inserting (\ref{e1}) and (\ref{e2}) into (\ref{an2}) and (\ref{bn2}), respectively.
\end{proof}
\begin{remark}
One will find that our expressions of $A_n(z)$ and $B_n(z)$ are very different from the results in \cite{ChenIts,Filipuk} due to the terms involving the recurrence coefficients $\al_n$ and $\bt_n$. This will lead to the fact that there are no simple relations between the recurrence coefficients $\al_n$ and $\bt_n$ and the auxiliary quantities $R_n(t)$ and $r_n(t)$ as shown in \cite{ChenIts,Filipuk} by using the compatibility conditions (see the following analysis in detail). In this case, we are not able to derive the second-order differential equations satisfied by the recurrence coefficients or the auxiliary quantities. As a consequence, the relations between our problem and the Painlev\'{e} equations are unclear.
\end{remark}
\begin{proof}[$\mathbf{Proof\; of\; Theorem\; \ref{thm}}$]
Substituting (\ref{an1}) and (\ref{bn1}) into ($S_{1}$) and comparing the coefficients of $\frac{1}{z}$ and $\frac{1}{z^{2}}$ on both sides, we find
\be\label{m1}
r_{n+1}(t)+r_{n}(t)=t-\alpha_{n}R_{n}(t),
\ee
\be\label{m2}
2\beta_{n+1}+2\beta_{n}=2n+1+\lambda+R_{n}(t)-2\alpha_{n}^2.
\ee
Similarly, substituting (\ref{an1}) and (\ref{bn1}) into ($S_{2}'$), we obtain
\be\label{s1}
r_{n}(t)\big(r_{n}(t)-t\big)=\beta_{n}R_{n}(t)R_{n-1}(t),
\ee
\be\label{s2}
(4\beta_{n}-2n-\la)r_{n}(t)+nt=2\beta_{n}\big(t+\al_{n}R_{n-1}(t)+\al_{n-1}R_{n}(t)\big),
\ee
\be\label{s3}
4\beta_{n}^2-2(2n+\la)\bt_n+n(n+\lambda)+\sum_{j=0}^{n-1}R_{j}(t)=2\beta_{n}\big(R_{n}(t)+R_{n-1}(t)+2\al_{n}\al_{n-1}\big),
\ee
\be\label{s4}
r_{n}(t)+\sum_{j=0}^{n-1}\al_j=2\beta_{n}(\al_{n}+\al_{n-1}).
\ee

From (\ref{m2}) we can express $R_n(t)$ in terms of the recurrence coefficients,
\be\label{Rn}
R_n(t)=2 \alpha_n^2+2 \beta_n+2\beta_{n+1}-2 n-1-\la.
\ee
Using (\ref{Rn}) to eliminate $R_n(t)$ and $R_{n-1}(t)$ in (\ref{s2}), we get the expression of $r_n(t)$ in terms of the recurrence coefficients,
\bea\label{rn}
r_n(t)&=&\frac{1}{2n+\la -4 \beta_n}\big[n t-2t\bt_n-2\al_n \beta_n(2\alpha_{n-1}^2+2 \beta_n+2 \beta_{n-1}-2 n+1-\la)\no\\
&-&2\alpha_{n-1}\beta_n (2 \alpha_n^2+2 \beta_n+2 \beta_{n+1} -2 n-1-\la)\big].
\eea
Substituting (\ref{Rn}) and (\ref{rn}) into (\ref{s1}), we obtain (\ref{d1}).

To proceed, note that we have the fact
$$
\sum_{j=0}^{n-1}\al_{j}(t)=-\mathrm{p}(n,t),
$$
which follows from (\ref{al}). So, the sub-leading coefficient $\mathrm{p}(n,t)$ can be expressed as the following form by (\ref{s4}),
\be\label{pn}
\mathrm{p}(n,t)=r_{n}(t)-2\beta_{n}(\al_{n}+\al_{n-1}).
\ee
Then we have
\bea\label{al2}
\al_n&=&\mathrm{p}(n,t)-\mathrm{p}(n+1,t)\no\\
&=&r_{n}(t)-r_{n+1}(t)-2\beta_{n}(\al_{n}+\al_{n-1})+2\beta_{n+1}(\al_{n+1}+\al_{n}).
\eea
Eliminating $r_{n+1}(t)$ from the combination of (\ref{m1}) and (\ref{al2}) gives
$$
2r_n(t)+\al_n\big(R_n(t)-1\big)-t-2\beta_{n}(\al_{n}+\al_{n-1})+2\beta_{n+1}(\al_{n+1}+\al_{n})=0.
$$
Inserting (\ref{Rn}) and (\ref{rn}) into the above, we obtain (\ref{d2}).
\end{proof}
\begin{remark}
If one substitutes (\ref{Rn}) and (\ref{rn}) into (\ref{m1}) directly, then the \textbf{third}-order difference equation for the recurrence coefficients will be obtained. In addition, one can derive another \textbf{third}-order difference equation for the recurrence coefficients by using (\ref{s3}).
\end{remark}
\begin{proof}[$\mathbf{Proof\; of\; Theorem\; \ref{thm1}}$]
From (\ref{or}) we have
$$
\int_{0}^{\infty}P_{n}^2(x;t)w(x;t)dx=h_{n}(t)
$$
and
$$
\int_{0}^{\infty}P_{n}(x;t)P_{n-1}(x;t)w(x;t)dx=0.
$$
Taking derivatives with respect to $t$ gives
$$
h_n'(t)=-\int_{0}^{\infty}\frac{1}{x}P_{n}^2(x;t)w(x;t)dx
$$
and
$$
\frac{d}{dt}\mathrm{p}(n,t)=\frac{1}{h_{n-1}(t)}\int_{0}^{\infty}\frac{1}{x}P_{n}(x;t)P_{n-1}(x;t)w(x;t)dx,
$$
respectively.
It follows that
\be\label{eq1}
t\frac{d}{dt}\ln h_n(t)=-R_n(t)
\ee
and
\be\label{eq2}
t\frac{d}{dt}\mathrm{p}(n,t)=r_n(t).
\ee

Combining (\ref{eq1}) with (\ref{be}), we have
$$
t\frac{d}{dt}\ln \bt_n(t)=R_{n-1}(t)-R_n(t).
$$
That is,
$$
t\bt_n'(t)=\bt_n(R_{n-1}(t)-R_n(t)).
$$
Substituting (\ref{Rn}) into the above gives (\ref{dd2}).
On the other hand, from (\ref{eq2}) and (\ref{al}) we find
\be\label{al4}
t\al_n'(t)=r_n(t)-r_{n+1}(t).
\ee
Eliminating $r_n(t)-r_{n+1}(t)$ from (\ref{al4}) and (\ref{al2}), we arrive at (\ref{dd1}). This completes the proof.
\end{proof}
\begin{proof}[$\mathbf{Proof\; of\; Theorem\; \ref{thm2}}$]
By using Dyson's Coulomb fluid approach and following similar discussions in \cite{Min2021,Min2022}, we find that the recurrence coefficients $\al_n$ and $\bt_n$ have the large $n$ expansion forms
\begin{subequations}\label{ab2}
\be\label{al3}
\al_n=\sqrt{\frac{2n}{3}}+\sum_{j=0}^{\infty}\frac{a_j}{n^{j/2}},
\ee
\be\label{be3}
\bt_n=\frac{n}{6}+\sum_{j=-1}^{\infty}\frac{b_j}{n^{j/2}}.
\ee
\end{subequations}
Substituting (\ref{ab2}) into the discrete system (\ref{de}), we obtain the expansion coefficients as $n$ gets large as follows:
\bea
&&a_0=0,\qquad a_1=\frac{2 \lambda+2 +3 \sqrt[3]{2t^2}}{4 \sqrt{6}},\qquad a_2=0,\qquad a_3=-\frac{4 \sqrt[3]{2} (3 \lambda +2)+18(\lambda +1) (2t)^{2/3}+27 t^{4/3}}{96\sqrt{3}\times 2^{5/6}},\no\\[8pt]
&&a_4=\frac{\lambda (\lambda ^2-1)}{72\times\sqrt[3]{4t}},\qquad b_{-1}=0,\qquad b_0=\frac{\lambda }{12},\qquad b_1=-\frac{\lambda  \sqrt[3]{t}}{4\sqrt{3}\times2^{5/6}},\qquad b_2=\frac{1-3 \lambda ^2}{144},\no\\[8pt]
&&b_3=-\frac{\lambda  \left[2^{2/3} (\lambda ^2-1)-6 \lambda  t^{2/3}-9 \sqrt[3]{2}\: t^{4/3}\right]}{96\sqrt{3}\times2^{5/6}   \sqrt[3]{t}}.
\eea
The theorem is then established.
\end{proof}
\begin{remark}
One can also derive the large $n$ asymptotic expansion for the sub-leading coefficient $\mathrm{p}(n,t)$ from (\ref{pn}) and (\ref{rn}) by using the results in Theorem \ref{thm2}.
\end{remark}

\section*{Acknowledgments}
The work of the first author was partially supported by the National Natural Science Foundation of China under grant number 12001212 and by the Fundamental Research Funds for the Central Universities under grant number ZQN-902. The work of the third author was partially supported by the Macau Science and Technology Development Fund under grant number FDCT 0079/2020/A2.



\vspace{58pt}
\noindent \textsc{School of Mathematical Sciences, Huaqiao University, Quanzhou 362021, China}\\

\noindent \textsc{Department of Mathematics, Faculty of Science and Technology, University of Macau, Macau, China}

\end{document}